\documentclass[12pt]{amsart}

\usepackage[utf8]{inputenc}
\usepackage{a4wide}
\usepackage{mathtools}
\usepackage[T1]{fontenc}
\usepackage{tcolorbox}
\usepackage{enumitem}
\usepackage{amsmath,amsthm,amssymb}
\usepackage[english,french]{babel}
\usepackage{listings}
\lstdefinelanguage{Sage}[]{Python}
{morekeywords={False,sage,True},sensitive=true}
\definecolor{dblackcolor}{rgb}{0.0,0.0,0.0}
\definecolor{dbluecolor}{rgb}{0.01,0.02,0.7}
\definecolor{dgreencolor}{rgb}{0.2,0.4,0.0}
\definecolor{dgraycolor}{rgb}{0.30,0.3,0.30}

\usepackage{mathrsfs}
\usepackage[bookmarks=true,backref,colorlinks=true,citecolor=blue,urlcolor=blue,linkcolor=magenta]{hyperref}

\usepackage{pdfcomment}			%
\usepackage[inline]{asymptote}	

\usepackage{enumitem}
\usepackage{tikz-cd}
\usepackage{amsrefs}
\usepackage{tikz}
\usepackage{tikz-cd}
\usetikzlibrary{%
  matrix,%
  calc,%
  arrows%
}

\usepackage{xcolor}
\usepackage{float}
\restylefloat{table}

\newcommand\OO{\mathcal{O}}
\newcommand\places{\Omega_K}
\DeclareMathOperator\norm{N}

\title[Abelian extensions with prescribed norms]{Constructing abelian extensions with prescribed norms}
\date{14 December 2020}

\author{Christopher~\textsc{Frei}}\address{TU Graz\\Institute of Analysis and
  Number Theory\\Steyrergasse 30/II\\8010 Graz, \textsc{Austria}}
\email{frei@math.tugraz.at}

\author{Rodolphe~\textsc{Richard}}
\address{Rodolphe \textsc{Richard}\\
13, rue du Croisic\\
22200 Plouisy,
Bretagne, \textsc{France}
}
\email{rodolphe.richard@normalesup.org}



\newtheorem{lemma}{Lemma}

\newtheorem{theorem}[lemma]{Theorem}

\theoremstyle{definition}

\newtheorem{definition}[lemma]{Definition}

\newcommand{\Z}{\mathbf{Z}}
\newcommand{\Q}{\mathbf{Q}}
\newcommand{\F}{\mathbf{F}}

\newcommand{\N}{\mathbf{N}}

\newcommand{\Gal}{\mathrm{Gal}}
\newcommand{\Pic}{\mathrm{Pic}}

\newcommand{\knew}{k}
\newcommand{\kprimenew}{k'}
\newcommand{\Phinew}{\Phi}
\newcommand{\Phiprimenew}{\Phi'}

\DeclareMathOperator\spl{Spl}

\begin{document}
\selectlanguage{english} 

\begin{abstract}
  Given a number field $K$, a finite abelian group $G$ and finitely many elements $\alpha_1,\ldots,\alpha_t\in K$, we construct abelian extensions $L/K$ with Galois group $G$ that realise all of the elements $\alpha_1,\ldots,\alpha_t$ as norms of elements in $L$. In particular, this shows existence of such extensions for any given parameters.

  Our approach relies on class field theory and a recent formulation of Tate's characterisation of the Hasse norm principle, a local-global principle for norms. The constructions are sufficiently explicit to be implemented on a computer, and we illustrate them with concrete examples.  
\end{abstract}

\maketitle
\tableofcontents
\section{Introduction}
Attached to each extension $L/K$ of number fields comes the field-theoretic norm map $\norm_{L/K}:L^\times \to K^\times$. Given an extension $L/K$, a classical problem is to study which elements of $K$ are in the image of this norm map. We consider the inverse problem: given elements $\alpha_1,\ldots,\alpha_t\in K^\times$, we construct an extension $L/K$ that realises all of them as norms. 

Of course the trivial extension does the trick, so we prescribe moreover a given degree and even a given Galois group $G$. It is easy to prove the existence of a degree-$n$-extension $L/K$ whose normal closure has full Galois group $S_n$, such that a given $\alpha\in K^\times$ is a norm: one may adjoin to $K$ a root of a polynomial $X^n+a_{n-1}X^{n-1}+\cdots+a_1X+(-1)^n\alpha$, with coefficients $a_1,\ldots,a_{n-1}$ sufficiently generic in~$K$ so that Hilbert's irreducibility theorem applies.

On the opposite end, one can look at Abelian extensions. In this situation,  D.~Loughran, R.~Newton and the first-named author have recently proved the following result.

\begin{theorem}{\cite[Theorem 1.1]{Frei}}\label{thm:existence}
  Let $K$ be a number field, $G$ a finite abelian group and $\alpha_1,\ldots,\alpha_t\in K^\times$. Then there is a normal extension $L/K$ with $\Gal(L/K)\simeq G$ and  $\{\alpha_1,\ldots,\alpha_t\}\subset\norm_{L/K}(L^\times)$.
\end{theorem}

In \cite{Frei}, this follows from an analytic result regarding the density of such extensions. The appendix to \cite{Frei} by Y.~Harpaz and O.~Wittenberg provides in addition an algebro-geometric proof of Theorem \ref{thm:existence}, using descent, a version of the fibration method developed in their work \cite{HW18}, and a version of Hilbert's irreducibility theorem.
\smallskip

Both of these proofs are not constructive, and the main achievement of our work is the direct construction of an extension $L/K$ that satisfies the conclusion of Theorem \ref{thm:existence}. Our construction is based on explicit class field theory, the availability of places of $K$ that satisfy certain Chebotarev-type conditions, and a version of Tate's criterion for the Hasse norm principle (see Theorem \ref{thm:HNP} in \S \ref{sec:HNP}), which was formulated in \cite{HasseNormPrinciple,Frei} and is of central importance for the quantitative arguments in these papers. Some of our constructions are reminiscent of arguments appearing in \cite{Jehne}.
\smallskip

A stronger version of Theorem \ref{thm:existence} is proved in \cite[Corollary 4.11]{Frei}. There, one can even specify finitely many places at which the extension $L/K$ is required to have a prescribed admissible local structure. Our constructions are sufficiently general to permit the same restrictions, thus also recovering \cite[Corollary 4.11]{Frei}. See \S \ref{sec:local} for precise statements.
\smallskip

Our constructions are explicit enough to be implemented on a computer. In \S \ref{sec:computation}, we provide some remarks on how this can be done, though one should note that \emph{efficiency} is not a central focus of this work. Nevertheless, we illustrate our constructions in \S \ref{sec:illustrations} with concrete examples obtained by hand and aided by computer.

\subsection*{Outline of the construction}
By class field theory, we can specify the extension $L/K$ through a
continuous epimorphism~$\rho$ from the id\`ele class group~$I_K/K^\times$ of $K$ to $G$.
We construct the epimorphism~$\rho$ in such a way that $\alpha_1,\ldots,\alpha_t$ are local norms at all places, and moreover $L/K$ satisfies the \emph{Hasse norm principle}, a local global principle for norms reviewed in \S \ref{sec:HNP}.
We rely
on
a criterion for the Hasse norm principle to hold,
in terms of the decomposition groups; see Theorem \ref{thm:HNP}.
This is especially amenable to class field theory and allows us to reformulate the desired conditions, precisely~\eqref{eq:cond_local_norms}--\eqref{eq:cond_psi}, in terms of conditions on the local components~$\rho_v$ of the epimorphism~$\rho$, namely~\eqref{eq:rho_local_norms}--\eqref{eq:rho_triv_S}.

The epimorphism~$\rho$
will be
induced by a simpler object, which we call a \emph{characteristic morphism}, introduced in \S \ref{sec:char_morph}. The former is defined on classes of id\`eles, and the latter on actual integral $S$-id\`eles, for a conveniently chosen set of places~$S$.
Lemmas~\ref{lemma} and~\ref{eq:rho_local_norms} explain explicitly how the characteristic morphism inducing $\rho$ determines the corresponding decomposition subgroups $G_v=\rho_{v}({K_{v}}^\times)\leq G$ of the extension~$L/K$.

We actually study in~\S\ref{sec:local_norms} a very particular class of characteristic morphisms, of the form~$\rho^T_S$ of~\eqref{characteristic morphisms type T}, which are are explicitly determined by the data of a finite set~$T$ of places~$v_i$ of~$K$. To guarantee that $\alpha_1,\ldots,\alpha_t$ are local norms at all places, we give elementary sufficient conditions~\eqref{eq:local_norm_elementary}, to be checked only at these places~$v_i$. This amounts to
Chebotarev-type conditions on the places~$v_i$, ensuring an unlimited supply of places to choose from.

In order to also satisfy the Hasse norm principle, we use
the
mentioned criterion,
Theorem \ref{thm:HNP},
which
requires us
to produce sufficiently many large decomposition groups, and a ``spanning type'' property of these. At this point it will help to work with auxiliary characteristic morphisms~$\rho'_S$ into an auxiliary group~$G'\simeq(\Lambda^2 G)^2$, related to~$G$ by an explicit map~$\Psi:G'\to G$ given in~\eqref{eq:def_Psi}. In~\S\ref{sec:forcing_HNP} we devise a concrete condition~\eqref{eq:basis_condition} on this auxiliary characteristic morphism~$\rho'_S$
to ensure the criterion of Theorem \ref{thm:HNP}
for the characteristic morphism~$\Psi\circ\rho'_S$. It involves places~$v_i$ and
auxiliary
places~$w_i$. The
relevant
decomposition groups
are those
at
the places~$v_i$, and the auxiliary place~$w_i$ is chosen to ensure
that
the decomposition group~${G'}_{v_i}\leq G'$ will be large enough.
The condition~\eqref{eq:basis_condition} has then to do with relative position of these decomposition groups~$G'_{v_i}$, guaranteeing the ``spanning type'' property
for
~$\Psi\circ\rho'_S$. 

We finally explain in \S\ref{sec:choosing_places} how one can find a set of places satisfying~\eqref{eq:basis_condition}
through an incremental construction. Here, the auxiliary places $w_i$ have to satisfy Chebotarev-type conditions of a form studied earlier in \S\ref{sec:supply}. 
If~$G$ can be generated by~$\knew$ elements, then we can stop as soon as we
have produced
~$\kprimenew={\knew \choose 2}$ places~$v_i$, and as many of the corresponding~$w_i$.

Folding up, the set~$T=\{v_1;w_1;...;v_{\kprimenew};w_{\kprimenew}\}$ determines an auxiliary characteristic morphism~\eqref{characteristic morphisms type T} into~$G'$ which satisfies condition~\eqref{eq:basis_condition}, see Theorem~\ref{thm:choose_places}. Composing with~$\Psi$
yields
a characteristic morphism into~$G$ and an induced epimorphism~$\rho:I_K/K^\times\to G$,
which
corresponds to an
extension~$L/K$. By Theorem~\ref{thm:main_construction} this~$\rho$ will satisfy~\eqref{eq:rho_local_norms}--\eqref{eq:rho_triv_S}, and hence
we are done.

\smallskip
Our constructions rely on reduction maps modulo well-chosen places of $K$. We require these places to satisfy certain Chebotarev conditions, and their abundance is secured by the Chebotarev density theorem. In practice, we find suitable places through exhaustive search.  

\subsection*{Acknowledgements}
This collaboration was initiated when C.F.~presented the results of \cite{Frei} at the ninth WOMBL one-day meeting at University of Cambridge in 2019, and continued during a visit of R.R.~to give a seminar talk at The University of Manchester. We thank the organisers and hosting institutions for their hospitality. C.F.~was supported by EPSRC grant EP/T01170X/1. R.R.~was supported by ERC grant GeTeMo 617129, and Leverhulme Research Project Grant ``Diophantine problems related to Shimura varieties''.

\section{Constructive proof of Theorem \ref{thm:existence}}
\subsection{Notation}
We fix the number field $K$, the elements $\alpha_1,\ldots,\alpha_t\in
K^\times$ and the finite abelian group $G$ henceforth. We write $\places$ for
the set of all places of $K$. By $S$, we will always denote a finite subset of $\places$ that contains all archimedean places. For a non-archimedean place $v$, we write $K_v$ for the completion of $K$ at $v$, $\OO_v$ for the subring of $v$-adic integers, and $\OO_v^\times$ for the multiplicative group of $v$-adic units. The ring of $S$-integers of $K$ is defined as $\OO_S=K\cap\bigcap_{v\notin S}\OO_v$, and its unit group, the group of $S$-units, is $\OO_S^\times=K\cap\bigcap_{v\notin S}\OO_v^\times$. Here and in similar situations, the index $v\notin S$ is understood to run over $\places\smallsetminus S$. The symbol $v$ will also be used for various objects defining the place $v$, including the corresponding prime ideals in $\OO_S$ and $\OO_v$, and the exponential valuation at $v$. We write, for example $\alpha\bmod v$ for the image of $\alpha\in \OO_v$ in the residue field $\F_v$ at $v$. For any field $F$ and $e\in\N$, we will write  $F^{\times e}=(F^\times)^e$ for the group of non-zero $e$-th powers in $F$.

We will always assume that $S$ is large enough, in particular we will require that
\begin{equation}
  \label{eq:S}
  \begin{aligned}
    &S \text{ contains all archimedean places, }\\
    &\OO_S\text{ is a principal ideal domain, and}\\
    &\{\alpha_1,\ldots,\alpha_t\}\subset\OO_S^\times.
  \end{aligned}
\end{equation}

\subsection{Class field theory}
By global class field theory, extensions $L/K$ with an isomorphism $\Gal(L/K)\to G$ are parameterised
by epimorphisms $\rho : I_K/K^\times\to G$. Here, $I_K$ is the id\`ele group of $K$, in which $K^\times$ is embedded diagonally, and a morphism between topological groups is a continuous group homomorphism. We endow the finite group $G$ with the discrete topology. The extension given by such an epimorphism $\rho$ will be denoted by $L_\rho/K$. By the interplay between local and global class field theory, the local behaviour of the extension $L_\rho$ corresponding to $\rho$ at $v$ is described by the restriction $\rho_v$ of $\rho$ to the subgroup $K_v^\times\subset I_K/K^\times$, embedded by sending $\alpha\in K_v^\times$ to the class of the id\`ele $(1,\ldots,1,\alpha,1,\ldots)$ that has component $\alpha$ at place $v$ and $1$ at every other place.

\subsection{Hasse norm principle}\label{sec:HNP}
Our goal is to construct an epimorphism $\rho$ whose corresponding extension $L=L_\rho$ satisfies $\{\alpha_1,\ldots,\alpha_t\}\subset\norm_{L/K}(L^\times)$. Since it is hard to detect global norms directly, we rely on local norms and a suitable local-global principle. Recall that the norm $\norm_{L/K}:L^\times\to K^\times$ extends to a homomorphism $\norm_{L/K}:I_{L}\to I_K$. By definition, the \emph{Hasse norm principle} holds for the extension $L/K$, if
\begin{equation*}
\norm_{L/K}(L^\times)= K^\times\cap\norm_{L/K}(I_L).
\end{equation*}
We call the elements of the  right-hand set \emph{local norms at all places} of $K$, so the Hasse norm principle, if valid for $L/K$, allows us to detect global norms as local norms at all places. The Hasse norm theorem \cite{Hasse} asserts the validity of the Hasse norm principle when $\Gal(L/K)$ is \emph{cyclic}. However, it may fail already in the case of biquadratic extensions, with the first counter-example of $\Q(\sqrt{-3},\sqrt{13})/\Q$ also due to Hasse. 

We rely on the following criterion for the Hasse norm principle, which is based
on Tate's cohomological description of the knot group $(K^\times\cap\norm_{L/K}(I_L))/\norm_{L/K}(L^\times)$ (e.g.~\cite[\S 11.4]{Tate}).

\begin{theorem}[{\cite[Lemma 4.2]{Frei}}, {\cite[\S6]{HasseNormPrinciple}}]\label{thm:HNP} Let $G$ be a finite abelian group and $\rho:I_K/K^\times\to G$ a surjective homomorphism. Then the Hasse norm principle holds for the extension $L_\rho/K$ if and only if the natural map
\begin{equation}\label{LG criterion}
\bigoplus_{v\in\places}{\bigwedge}^2 \rho_v(K_v^\times) \to {\bigwedge}^2 G
\end{equation}
is surjective.
\end{theorem}

\subsection{Local conditions}\label{sec:local} For additional flexibility, we allow ourselves to impose arbitrary local conditions at the finite set of places $S$, as long as these conditions are realised by some sub-$G$-extension that admits $\alpha_1,\ldots,\alpha_t$ as local norms at all places. More precisely, assume we are given a morphism $\psi : I_K/K^\times \to G$, not necessarily surjective, such that $\{\alpha_1,\ldots,\alpha_t\}\subset \norm_{L_\psi/K}(I_{L_\psi})$. We will construct an epimorphism $\varphi : I_K/K^\times \to G$ with the following properties:
\begin{align}
&\text{$\{\alpha_1,\ldots,\alpha_t\}\subset\norm_{L_\varphi/K}(I_{L_\varphi})$,}\label{eq:cond_local_norms}\\
&\text{the extension $L_\varphi/K$ satisfies the Hasse norm principle, and}\label{eq:cond_HNP}\\
& \text{the local restriction $\varphi_v$ agrees with $\psi_v$ on $K_v^\times$ for all $v\in S$.}\label{eq:cond_psi}
\end{align}
In particular, by \eqref{eq:cond_local_norms} and \eqref{eq:cond_HNP} the $\alpha_i$ are global norms from $L_\varphi$, and thus we obtain a  constructive proof of Theorem \ref{thm:existence} and \cite[Corollary 4.11]{Frei}.

Possibly enlarging $S$, we will assume that it contains all places at which $\psi$ is ramified. Let $\spl_S(\psi)$ be the set of all places $v\notin S$ that are split completely in $L_\psi$, i.e.~$\psi_v(K_v^\times)=\{e_G\}$.

In the rest of this section, we will explain how to construct an epimorphism $\rho:I_K/K^\times\to G$ that satisfies the following:
\begin{align}
  &\text{for all $v\in\places$, }\{\alpha_1,\ldots,\alpha_t\}\subset\ker(\rho_v:K_v^\times\to I_K/K^\times\to G),\label{eq:rho_local_norms}\\
  &\text{the natural map }\bigoplus_{v\in\spl_S(\psi)}{\bigwedge}^2 \rho_v(K_v^\times) \to {\bigwedge}^2 G \text{ is surjective,}\label{eq:rho_HNP}\\
  &G=\sum_{v\in\spl_S(\psi)}\rho_v(K_v^\times)\text{, and}\label{eq:rho_surj}\\
  &\text{for all $v\in S$, the local morphism $\rho_v:K_v^\times\to G$ is trivial.}\label{eq:rho_triv_S} 
\end{align}

Once we have constructed such an epimorphism $\rho$, we can take $\varphi:I_K/K^\times \to G$ to be the product of $\psi$ and $\rho$. The following lemma justifies this.

\begin{lemma}\label{lem:twist}
If $\rho:I_K/K^\times\to G$ is an epimorphism with the properties \eqref{eq:rho_local_norms}--\eqref{eq:rho_triv_S}, then the product $\varphi = \psi\rho : I_K/K^\times\to G$ is an epimorphism with the properties \eqref{eq:cond_local_norms}--\eqref{eq:cond_psi}.
\end{lemma}

\begin{proof}
For the local morphisms at all places $v$, we have $\varphi_v = \psi_v\rho_v : K_v^\times\to G$. For $v\in\spl_S(\psi)$, the morphism $\psi_v$ is trivial, and hence $\varphi_v=\rho_v$. This shows that
\begin{equation*}
 G=\sum_{v\in\spl_S(\psi)}\varphi_v(K_v^\times)\quad\text{ and that }\quad \bigoplus_{v\in\spl_S(\psi)}{\bigwedge}^2 \varphi_v(K_v^\times) \to {\bigwedge}^2 G\text{ is surjective.}
\end{equation*}
In particular, $\varphi$ is surjective and  \eqref{eq:cond_HNP} holds by Theorem \ref{thm:HNP}. For $v\in S$, the local morphism $\rho_v$ is trivial and thus $\varphi_v = \psi_v$, as desired in \eqref{eq:cond_psi}. For each $\alpha_i$ and $v\in S$, we see by local class field theory that $\alpha_i\in\ker(\psi_v)=\ker(\varphi_v)$. Let $v\notin S$, then $\psi$ is unramified at $v$ and thus $\psi_v$ vanishes on $\OO_{v}^\times$, which contains all of the $\alpha_i$ by \eqref{eq:S}. Hence, $\varphi_v(\alpha_i)=\rho_v(\alpha_i)=e_G$. We have shown that all the local morphisms $\varphi_v$ vanish on all $\alpha_i$, hence \eqref{eq:cond_local_norms} holds by local class field theory.
\end{proof}

\subsection{Characteristic morphisms}\label{sec:char_morph}  Here we describe a hands-on way of constructing an epimorphism $\rho:I_K/K^\times\to G$ through the specification of a simpler object which we call a \emph{characteristic morphism}.

Consider the $S$-id\`eles $I_S = K_S^\times\times\widehat{\OO}_S^\times\subset I_K$, where
\begin{equation*}
K_S^\times = \prod_{v\in S}K_v^\times\quad \text{ and }\quad \widehat{\OO}_S^\times=\prod_{v\notin S}\OO_{v}^\times.
\end{equation*}
Since the $S$-class group $\Pic(\OO_S)$ is trivial by \eqref{eq:S}, we have $I_SK^\times = I_K$. Inclusion of the open subgroup $I_S$ into $I_K$ thus induces an open epimorphism $I_S \to I_K/K^\times$ with kernel $I_S\cap K^\times = \OO_S^\times$, and hence an isomorphism
\begin{equation}\label{eq:isom}
I_S/\OO_S^\times \cong I_K/K^\times.
\end{equation}
With $p: I_S \to \widehat{\OO}_S^\times$ the natural projection, we define the following.  
\begin{definition}
  A \emph{characteristic morphism} for $G$ is a continuous epimorphism $\rho_S : \widehat{\OO}_S^\times\to G$ that vanishes on $p(\OO_S^\times)$.
\end{definition}
Since $\rho_S$ is by definition continuous, it vanishes on the closure of $p(\OO_S^\times)$. Note that, by the Dirichlet $S$-unit theorem (e.g. \cite[Chapter VI, Prop.~1.1]{Neukirch}), $\OO_S^\times$ is finitely generated of $\Q$-rank $r=|S|-1$, so it has generators $\gamma_0,\ldots,\gamma_r$, with $\gamma_0$ a root of unity. For $\rho_S$ to vanish on $p(\OO_S^\times)$, it is thus enough that it vanishes on the images of $\gamma_0,\ldots,\gamma_r$ under the diagonal embedding $\OO_S^\times\to\widehat{\OO}_S^\times$.

A characteristic morphism $\rho_S:\widehat{\OO}_S^\times \to G$ induces an epimorphism $\rho: I_K/K^\times\to G$ as follows: as the epimorphism $\rho_S\circ p : I_S \to G$ vanishes on $\OO_S^\times$, it induces an epimorphism $I_S/\OO_S^\times \to G$, which by \eqref{eq:isom} gives the desired $\rho$.

We can also describe the restriction $\rho_v$ of $\rho$ to $K_v^\times\subset I_K/K^\times$ in terms of the characteristic morphism $\rho_S$. To do so, we need the following simple lemma.

\begin{lemma}\label{lemma} Every place $v\in\places\smallsetminus S$ admits a uniformiser $\pi_v\in \OO_{S\cup\{v\}}^\times$, and $\pi_v$ is unique up to multiplication by $S$-units.
\end{lemma}
\begin{proof}
The admissible choices for $\pi_v$ are the generators of the prime ideal $(\OO_v\smallsetminus\OO_v^\times)\cap\OO_S$ of $\OO_S$, which is principal due to \eqref{eq:S}. 
\end{proof}

For $v\notin S$, we embed $K_v^\times$ into $\widehat{\OO}_S^\times$ as follows: we identify $\OO_{v}^\times$ with the $v$-component of $\widehat{\OO}_S^\times$. In addition,  
we fix a uniformiser $\pi_v\in \OO_{S\cup\{v\}}^\times$ at $v$ as in Lemma \ref{lemma}, which we map to the element of $\widehat{\OO}_S^\times$ that is $1$ at its $v$-component and $\pi_v^{-1}\in \OO_{S\cup\{v\}}^\times\subset \OO_{w}^\times$ at the component indexed by $w\notin S\cup\{v\}$.

\begin{lemma}\label{lem:embedding}
  Let $\rho_S:\widehat{\OO}_S^\times\to G$ be a characteristic morphism and $\rho$ the induced epimorphism $I_K/K^\times \to G$. Then $\rho$ vanishes identically on $K_v^\times$ for $v\in S$. For $v\notin S$, let the embedding $K_v^\times\to \widehat{\OO}_S^\times$ be as specified above. Then the following diagram commutes:
\begin{center}
  \begin{tikzpicture}[>=triangle 60]
\matrix[matrix of math nodes,column sep={60pt,between origins},row
sep={60pt,between origins},nodes={asymmetrical rectangle}] (s)
{
&|[name=a]| K_v^\times &|[name=b]| I_K/K^\times \\
&|[name=c]| \widehat{\OO}_S^\times &|[name=d]| G  \\
};
\draw[->] (a) edge (b)
          (a) edge (c)
          (b) edge node[left] {$\rho$} (d)
          (c) edge node[above] {$\rho_S$} (d);
        \end{tikzpicture}
\end{center}
\end{lemma}

\begin{proof}
The key case is the image of $\pi_v$ for $v\notin S$. In this case,
\begin{equation*}
  (1,\ldots,1,\pi_v,1,\ldots)K^\times = (\pi_v^{-1},\ldots,\pi_v^{-1},1,\pi_v^{-1},\ldots)K^\times,
\end{equation*}
and the latter id\`ele is in $I_S$ by our choice of $\pi_v$. Hence,
\begin{equation*}
  \rho(\pi_v) = (\rho_S\circ p)((\pi_v^{-1},\ldots,\pi_v^{-1},1,\pi_v^{-1},\ldots)) = \rho_S(\pi_v),
\end{equation*}
by definition of the embedding of $K_v^\times$ in $\widehat{\OO}_S^\times$.
\end{proof}

\subsection{Local norms}\label{sec:local_norms}
Here we construct characteristic morphisms whose induced epimorphisms $I_K/K^\times\to G$ have the properties \eqref{eq:rho_local_norms}, \eqref{eq:rho_surj} and \eqref{eq:rho_triv_S}. In particular, $\alpha_1,\ldots,\alpha_t$ are local norms from the corresponding extensions at all places. When $G$ is cyclic, then ${\bigwedge}^2G=0$, so \eqref{eq:rho_HNP} holds automatically and we are done. We will return to \eqref{eq:rho_HNP} for non-cyclic $G$ in \S \ref{sec:forcing_HNP}.

For a non-archimedean place $v\in\places$, we write $\F_v$ for the residue field at $v$ and $q_v=|\F_v|$. Given $e\in\N$ and a finite set $\{x_0,\ldots, x_n\}\subseteq K^\times$, we denote by $T(S;e;\psi;\{x_0, \ldots, x_n\})$ the set of places $v$ of $K$ that satisfy the following:
\begin{equation}
  \label{eq:local_norm_elementary}
  \begin{aligned}
	& v\in\spl_S(\psi),\\
	& q_v\equiv 1\bmod e,\\
	& \{x_0,\ldots,x_n\}\subset \OO_v^\times,\\
	& x_0,\ldots,x_n\text{ are $e$-th powers in the residue field }\F_v^\times.
 \end{aligned}
\end{equation}
Note that $T(S;e;\psi;\{x_0,\ldots,x_n\})$ contains precisely those places in $\places\smallsetminus S$ that split completely in the normal extension $L_\psi(\zeta_e,\sqrt[e]{x_0},\ldots,\sqrt[e]{x_n})/K$, where $\zeta_e$ is a primitive $e$-th root of unity in $\bar{K}$. Hence, by the Chebotarev density theorem, $T(S;e;\psi;\{x_0,\ldots,x_n\})$ is infinite and of density $1/[L_\psi(\zeta_e,\sqrt[e]{x_0},\ldots,\sqrt[e]{x_n}):K]$ in $\places$. Every $v\in T(S;e;\psi;\{x_0,\ldots,x_n\})$ satisfies in particular that
\begin{equation}\label{eq:F_v_isom}
\F_v^\times\otimes\Z/e\Z \simeq \Z/e\Z,
\end{equation}
as $e\mid |\F_v^\times|$. Now let $e$ be a multiple of the
exponent of $G$ and $k$ large enough, so we can fix an epimorphism
$\Phi:(\Z/e\Z)^k\to G$.  Recall that $\{\gamma_0,\ldots,\gamma_r\}$ is a set
of generators of $\OO_S^\times$. For any subset $T=\{v_1,\ldots,v_k\}\subset T(S;e;\psi; \{\gamma_0,\ldots,\gamma_r\})$ of $k$ places, we define the epimorphism
\begin{equation}\label{characteristic morphisms type T}
\rho_S^T:
\widehat{\OO_S}^\times\to \F_{v_1}^\times\times\cdots\times \F_{v_k}^\times\to  (\F_{v_1}^\times\times\ldots\times \F_{v_k}^\times)\otimes \Z/e\Z\xrightarrow{\simeq} (\Z/e\Z)^k\xrightarrow{\Phi} G,
\end{equation}
where the first two maps are the natural epimorphisms, and the third one comes from a choice of the isomorphisms \eqref{eq:F_v_isom}.

\begin{lemma}\label{lem:rho_S_T}
  The morphism $\rho_S^T$ defined in \eqref{characteristic morphisms type T} is a characteristic morphism, and the induced epimorphism $\rho^T:I_K/K^\times\to G$ satisfies \eqref{eq:rho_local_norms}, \eqref{eq:rho_surj} and \eqref{eq:rho_triv_S}.
\end{lemma}

\begin{proof}
  For $\gamma\in K$ and $v\notin S$, write $\gamma_v$ for the local embedding $\gamma\in K_v^\times\subset\widehat{\OO}_S^\times$. We start by showing that $\gamma_v\in\ker\rho_S^T$ for all $v\notin S$ and $\gamma\in\{\gamma_0,\ldots,\gamma_r\}$. If $v\notin T$, this is obviously true by the construction of $\rho_S^T$. For $v\in T$, our choice of $T$ ensures that $\gamma$ is an $e$-th power in $\F_v^\times$, and hence vanishes in $\F_v^\times\otimes\Z/e\Z$.

  Since $\gamma_0,\ldots,\gamma_r$ generate $\OO_S^\times$, we have shown that $\rho_S^T$ vanishes at all local components $\gamma_v$, for $v\notin S$ and $\gamma\in\OO_S^\times$. This implies in particular that $p(\OO_S^\times)\subset \ker\rho_S^T$, so $\rho_S^T$ is a characteristic morphism. Moreover, as $\{\alpha_1,\ldots,\alpha_t\}\subset\OO_S^\times$ by \eqref{eq:S} and since the induced epimorphism $\rho^T$ vanishes on all of $K_v^\times$ for $v\in S$, we immediately get \eqref{eq:rho_local_norms} and \eqref{eq:rho_triv_S}. By the construction of $\rho_S^T$, it is also obvious that the condition \eqref{eq:rho_surj} is satisfied.
\end{proof}

\subsection{Forcing validity of the Hasse norm principle}\label{sec:forcing_HNP} Here we discuss how to choose the morphism $\rho_S^T$ from the previous section in order to guarantee that the induced epimorphism $\rho^T$  satisfies \eqref{eq:rho_HNP}. We will assume that $G$ is not cyclic, as otherwise there is nothing to do. 

Fix an epimorphism $\Phinew:(\Z/e\Z)^{\knew}\to G$, with some $\knew\geq 2$. Write $\kprimenew={\knew\choose 2}$ and take a set of $2\kprimenew$ distinct places $T=\{v_1,w_1,v_2,w_2,\ldots,v_{\kprimenew},w_{\kprimenew}\}$ such that
\begin{equation*}
  \{v_1,\ldots,v_{\kprimenew}\}\subset T(S;e;\psi;\{\gamma_0,\ldots,\gamma_r\})\text{ and }  \{w_1,\ldots,w_{\kprimenew}\}\subset T(S;e;1;\{\gamma_0,\ldots,\gamma_r\}).
\end{equation*}
We write 
\begin{equation*}
  G'=(\F_{v_1}^\times\times \F_{w_1}^\times\times\cdots\times\F_{v_{\kprimenew}}^\times\times\F_{w_{\kprimenew}}^\times)\otimes \Z/e\Z\simeq ((\Z/e\Z)^2)^{\kprimenew}
\end{equation*}
and consider the auxiliary morphism
\begin{equation}\label{eq:rho_prime}
  \rho_S':\widehat{\OO}_S^\times\to \F_{v_1}^\times\times \F_{w_1}^\times\times\cdots\times\F_{v_{\kprimenew}}^\times\times\F_{w_{\kprimenew}}^\times\to G'
\end{equation}
as in \eqref{characteristic morphisms type T}. (Formally, to fit the definition of $\rho_S^T$ in \eqref{characteristic morphisms type T}, take $(\Z/e\Z)^{2\kprimenew}\to G'$ to be the inverse of the isomorphism coming from \eqref{eq:F_v_isom}.) By Lemma \ref{lem:rho_S_T} (with $\psi=1$), this is a characteristic morphism for $G'$, so it induces an epimorphism $\rho':I_K/K^\times\to G'$. We will show later in \S \ref{sec:choosing_places} that the places in $T$ can be chosen in such a way that the following holds:
\begin{equation}\label{eq:basis_condition}
  \parbox{0.8\textwidth}{There is a $\Z/e\Z$-basis $e_1,e_1',\ldots,e_{\kprimenew},e_{\kprimenew}'$ of $G'$, such that $e_i,e_i'$ is a basis of the decomposition group $G_{v_i}'=\rho'(K_{v_i}^\times)$ for all $1\leq i\leq \kprimenew$.}
\end{equation}
Let $e_1,e_1',\ldots,e_{\kprimenew},e_{\kprimenew}'$ be a basis of $G'$ as in \eqref{eq:basis_condition}, and choose any basis $f_1,\ldots,f_{\knew}$ of $(\Z/e\Z)^{\knew}$, and any enumeration $(f_{m_i},f_{n_i})_{1\leq i\leq \kprimenew}$ of the $\kprimenew={\knew\choose 2}$ pairs $(f_m,f_n)$ with $1\leq m<n\leq \knew$. We define the epimorphism of $\Z/e\Z$-modules $\Psi:G'\to (\Z/e\Z)^{\knew}$ by
\begin{equation}\label{eq:def_Psi}
  \Psi(e_i)=f_{m_i} \text{ and }\Psi(e_i')=f_{n_i} \text{ for }1\leq i\leq \kprimenew,
\end{equation}
and correspondingly the epimorphism $\Phiprimenew = \Phinew\circ\Psi : G'\to G$. With this $\Phiprimenew$ and the set $T$ from above, we define the morphism
\begin{equation*}
  \rho_S^T = \Phiprimenew\circ\rho_S':\widehat{\OO_S}^\times\xrightarrow{\rho_S'} G' \xrightarrow{\Phiprimenew} G.
\end{equation*}

\begin{theorem}\label{thm:main_construction}
  Assume that \eqref{eq:basis_condition} holds. Then the morphism $\rho_S^T:\widehat{\OO}_S^\times\to G$ defined above is a characteristic morphism. The induced epimorphism $\rho=\rho^T:I_K/K^\times\to G$ satisfies \eqref{eq:rho_local_norms}--\eqref{eq:rho_triv_S}. In particular, the extension $L_{\rho}/K$ has Galois group $G$ and $\{\alpha_1,\ldots,\alpha_t\}\subset\norm_{L_{\rho}/K}(L_{\rho}^\times)$. 
\end{theorem}

\begin{proof}
  Since $\rho_S^T$ is of the form \eqref{characteristic morphisms type T}, Lemma \ref{lem:rho_S_T} (with $\psi=1$) shows that $\rho_S^T$ is a characteristic morhism and that $\rho$ satisfies \eqref{eq:rho_local_norms} and \eqref{eq:rho_triv_S}. For the induced morphisms, note that $\rho=\Phiprimenew\circ\rho'$. Since $v_1,\ldots,v_{\kprimenew}\in\spl_S(\psi)$, condition \eqref{eq:rho_surj} follows immediately from our assumption \eqref{eq:basis_condition} and the fact that $\Phiprimenew$ is surjective.

  It remains to verify \eqref{eq:rho_HNP}. Since $f_1\ldots,f_{\knew}$ generate $(\Z/e\Z)^{\knew}$ and $\Phinew:(\Z/e\Z)^{\knew} \to G$ is an epimorphism, the group ${\bigwedge}^2G$ is generated by the elements $\Phinew(f_{m_i})\wedge \Phinew(f_{n_i})$ for $1\leq i\leq \kprimenew$. By our hypothesis \eqref{eq:basis_condition}  and the definition of $\Psi$ in \eqref{eq:def_Psi}, we see that $\{\Phinew(f_{m_i}),\Phinew(f_{n_i})\}\subset\rho_v(K_{v_i}^\times)$. Hence, $\Phinew(f_{m_i})\wedge\Phinew(f_{n_i})$ is in the image of the natural map ${\bigwedge}^2\rho_{v_i}(K_{v_i}^\times)\to{\bigwedge}^2G$ for all $1\leq i\leq \kprimenew$.
This is enough to show that
\begin{equation*}
  \bigoplus_{1\leq i\leq \kprimenew}{\bigwedge}^2 \rho_{v_i}(K_{v_i}^\times) \to {\bigwedge}^2G
\end{equation*}
is surjective. As $v_1,\ldots,v_{\kprimenew}\in\spl_S(\psi)$, we have proved \eqref{eq:rho_HNP}.
\end{proof}

\subsection{Supply of places}\label{sec:supply} We still need to specify how to choose the set $\{v_1,w_1\ldots,v_{\kprimenew},w_{\kprimenew}\}$ of places from \S \ref{sec:forcing_HNP} in order to guarantee that $\rho'$ satisfies \eqref{eq:basis_condition}.

For $e\in\N$, a finite set $\{x_0,\ldots,x_n\}\subset K^\times$ and $y\in K^\times$, we define $T(S;e;\{x_0,\ldots,x_n\}; y)$ as the set of all places $w\in T(S;e;1;\{x_0,\ldots,x_n\})$ such that $y\in\OO_w^\times$ and the polynomial $X^e-y$ modulo $w$ is irreducible over the residue field $\F_w$.

Let us show that these sets provide an ample supply of places in those cases that will be of interest to us.

\begin{lemma}\label{lem:cheb}
  Suppose there is a place $v\in T(S;e;1;\{x_0,\ldots,x_n\})$ such that $y$ is a uniformiser at $v$. Then the set $T(S;e;\{x_0,\ldots,x_n\}; y)$ is infinite and of natural density
  \begin{equation*}
\frac{\phi(e)}{e[K(\zeta_e,\sqrt[e]{x_1},\ldots,\sqrt[e]{x_n}):K]}>0
\end{equation*}
  in the set of all places $v$ of $K$ when ordered by $q_v$.
\end{lemma}

\begin{proof}
  Write $F = K(\zeta_e,\sqrt[e]{x_1},\ldots,\sqrt[e]{x_n})$. Since $v$ splits completely in $F$, the element $y$ is a uniformiser for ever place $\tilde{v}$ of $F$ above $v$. Therefore, $f=X^e-y$ is irreducible over $F$ by Eisenstein's criterion in $F_{\tilde{v}}$, and hence the extension $F(\sqrt[e]{y})/F$ is cyclic of degree $e$.

The set $T(S;e;\{x_0,\ldots,x_n\}; y)$ is precisely the set of places $w\in \spl_S(F)$, such that all places $\tilde{w}$ of $F$ above $w$ are inert in $F(\sqrt[e]{y})$.

These are precisely the places $w\notin S$ that do not divide $y$ and whose Frobenius class in $\Gal(F(\sqrt[e]{y})/K)$ consists of generators of the normal subgroup $\Gal(F(\sqrt[e]{y})/F)\simeq \Z/e\Z$. The desired result follows from the Chebotarev density theorem.
\end{proof}

\begin{lemma}\label{lem:irreducibility_enough}
  If $w\in T(S;e;\{x_0,\ldots,x_n\};y)$, then $(y\bmod w)\otimes 1$ generates $\F_w^\times\otimes\Z/e\Z$.
\end{lemma}

\begin{proof}
  Let $a\in \F_w^\times$. Then, $a\otimes 1$ generates $\F_w^\times\otimes\Z/e\Z$ if and only if none of the elements $a,a^2,\ldots,a^{e-1}$ are in $\F_w^{\times e}$. But $a^j\in\F_w^{\times e}$ if and only if $a\in\F_w^{\times e/\gcd(j,e)}$, and thus $a\otimes 1$ generates $\F_w^{\times e}$ if and only if $a\notin \F_w^{\times p}$ for all primes $p\mid e$. The latter condition is clearly satisfied if $X^e-a$ is irreducible over $\F_w$.
\end{proof}

\subsection{Choosing the right places}\label{sec:choosing_places}
With our supply of places guaranteed, we choose the the places $v_1,w_1,\ldots,v_{\kprimenew},w_{\kprimenew}$ as follows. Take $v_1,\ldots,v_{\kprimenew}$ to be any $\kprimenew$ distinct elements of $T(S;e;\psi;\{\gamma_0,\ldots,\gamma_r\})$. We fix a uniformiser $\pi_i\in\OO_{S\cup\{v_i\}}^\times$ for each $v_i$, which exists by Lemma \ref{lemma}.

Moreover, we choose each place $w_i\in T(S;e;\{\gamma_0,\ldots,\gamma_r\}; u_i\pi_i)\smallsetminus\{v_1,\ldots,v_{\kprimenew}\}$, with certain local units $u_i\in K\cap\OO_{v_i}^\times$, to be determined inductively. As $u_i\pi_i$ is a uniformiser at $v_i$, the hypothesis of Lemma \ref{lem:cheb} is satisfied.

Recall from the discussion leading up to Lemma \ref{lem:embedding} that each $K_{v_i}^\times$ is embedded into $\widehat{\OO}_S^\times$ by sending $\OO_{v_i}^\times$ to its factor in $\widehat{\OO}_S^\times$, and $\pi_i$ to $(\pi_i^{-1},\ldots,\pi_i^{-1},1,\pi_i^{-1},\ldots)$.
We will now determine the $u_i$ and $w_i$ inductively, such that the following property holds for all $i$:
\begin{equation}\label{eq:induction}
  \parbox{0.7\textwidth}{The images of $\pi_1,\ldots,\pi_i\in\widehat{\OO}_S^\times$ in $(\F_{w_1}^\times\times\cdots\times\F_{w_i}^\times)\otimes\Z/e\Z$ generate $(\F_{w_1}^\times\times\cdots\times\F_{w_i}^\times)\otimes\Z/e\Z$.} 
\end{equation}

For $i=1$, we take $u_1=1$ and choose $w_1\in T(S;e;\{\gamma_0,\ldots,\gamma_r\}; \pi_1)\smallsetminus\{v_1,\ldots,v_{\kprimenew}\}$. Then the image of $\pi_1$ under
\begin{equation*}
\OO_{S\cup\{v_1\}}^\times\to \OO_{w_1}^\times\to\F_{w_1}^\times\to\F_{w_1}^\times\otimes\Z/e\Z
\end{equation*}
generates $\F_{w_1}^\times\otimes\Z/e\Z$, as $X^e-\pi_1$ is irreducible over $\F_{w_1}$. Therefore, we see that $\F_{w_1}^\times\otimes\Z/e\Z$ is also spanned by the image $(\pi_1^{-1}\bmod w_1)\otimes 1$ of $\pi_1\in K_{v_1}\subset\widehat{\OO}_S^\times$ under $\widehat{\OO}_S^\times \to \F_{w_1}^\times\otimes\Z/e\Z$.

For $j>1$, assume that we have determined $u_1,\ldots,u_{j-1}$ and chosen $w_1,\ldots,w_{j-1}$ such that \eqref{eq:induction} holds for $i=j-1$. Then there are $c_1,\ldots,c_{j-1}\in\{0,\ldots,e-1\}$ such that the elements $\pi_j$ and $\pi_1^{c_1}\cdots \pi_{j-1}^{c_{j-1}}$ of $\widehat{\OO}_S^\times$ have the same image in $(\F_{w_1}^\times\times\cdots\times\F_{w_{j-1}}^\times)\otimes\Z/e\Z$.

We set $u_j=\pi_1^{-c_1}\cdots \pi_{j-1}^{-c_{j-1}}\in K\cap\OO_{v_j}^{\times}$ and choose $w_j\in T(S;e;\{\gamma_0,\ldots,\gamma_r\}; u_j\pi_j)\smallsetminus\{v_1,\ldots,v_{\kprimenew},w_1,\ldots,w_{j-1}\}$. Consider the projection
\begin{equation*}
   P : (\F_{w_1}^\times\times\cdots\times\F_{w_{j}}^\times)\otimes\Z/e\Z \to (\F_{w_1}^\times\times\cdots\times\F_{w_{j-1}}^\times)\otimes\Z/e\Z.
\end{equation*}
Then the images of $\pi_1,\ldots,\pi_{j-1}\in\widehat{\OO}_S^\times$ generate the range of $P$ by our inductive hypothesis. By construction, the image of $\pi_1^{-c_1}\cdots\pi_{j-1}^{-c_{j-1}}\pi_j\in\widehat{\OO}_S^\times$ is in $\ker P$. It even generates $\ker P$, since $X^e-u_j\pi_j$ is irreducible over $\F_{w_j}$, so $((u_j\pi_j)^{-1}\bmod w_j)\otimes 1$  generates $\F_{w_j}^\times\otimes\Z/e\Z$. This shows that \eqref{eq:induction} holds for $i=j$, as desired.

Finally, let us show that the places chosen above achieve what we want to satisfy the hypothesis of Theorem \ref{thm:main_construction}.

\begin{theorem}\label{thm:choose_places}
  Let the set $T=\{v_1,w_1,\ldots,v_{\kprimenew},w_{\kprimenew}\}\subset T(S;e;1;\{\gamma_0,\ldots,\gamma_r\})$ be chosen as described above. Then the epimorphism $\rho':I_K/K^\times\to G'$, induced by the auxiliary morphism $\rho_S'$ defined in \eqref{eq:rho_prime}, satisfies \eqref{eq:basis_condition}. 
\end{theorem}

\begin{proof}
  For each $1\leq i\leq \kprimenew$, we have $K_{v_i}^\times=\OO_{v_i}^\times\oplus \langle\pi_i\rangle$. Using Lemma \ref{lem:embedding}, we see that $\rho'(\OO_{v_i}^\times)$ is the factor $\F_{v_i}^\times\otimes\Z/e\Z\simeq\Z/e\Z$ of $G'$. Let $e_i$ be a generator of this factor.

  Moreover, take $e_i'$ to be the image of $\pi_i$ in $G'$. Then $e_i,e_i'$ is a basis of $\rho'(K_{v_i}^\times)$. Consider the projection $P : G' \to (\F_{w_1}^\times\times\cdots\times\F_{w_{\kprimenew}}^\times)\otimes\Z/e\Z$. Then clearly $e_1,\ldots,e_{\kprimenew}$ is a basis of $\ker P\cong (\F_{v_1}^\times\times\cdots\times\F_{v_{\kprimenew}})\otimes \Z/e\Z$. Moreover, due to \eqref{eq:induction}, the images $P(e_1'),\ldots,P(e_{\kprimenew}')$ form a basis of the image of $P$. 
\end{proof}

\section{Computations with characteristic morphisms}\label{sec:computation}
Let us discuss how to compute a characteristic morphism of the form $\rho_S^T$ as in \eqref{eq:rho_prime} and the map $\Phiprimenew:G'\to G$ in practice, and how to infer from this data some information about the $G$-extension defined by it. We will focus on the case where $\psi=1$, in order to avoid having to specify how $\psi$ is represented. We assume to be given the set $S$ of places satisfying \eqref{eq:S}, the  epimorphism $\Phinew:(\Z/e\Z)^{\knew}\to G$, and $\kprimenew={\knew\choose 2}$.

We need to find places $v_1,w_1,\ldots,v_{\kprimenew},w_{\kprimenew}$ such that \eqref{eq:induction} is satisfied for all $1\leq i\leq \kprimenew$.

\subsection{Finding the places $v_i$}\label{sec:finding_vi}
First, we need a set of generators $\{\gamma_0,\ldots,\gamma_r\}$ for the $S$-units $\OO_S^\times$, which can be computed by well-known algorithms implemented in many computer algebra systems. 

As $\psi=1$, the places $v_1,\ldots,v_{\kprimenew}$ have to be members of the set $T(S; e; 1; \{\gamma_0,\ldots,\gamma_r\})$, so they need to be in $\places\smallsetminus S$ and such that
\begin{equation}\label{eq:search_cond_v}
  q_v\equiv 1\bmod e \text{, and the residues of }\gamma_0,\ldots,\gamma_r\text{ are $e$-th powers in }\F_v^\times.
\end{equation}
These places have positive density among all places of $K$ by the Chebotarev density theorem, so we can produce enough of them by an exhaustive search. Effective lower bound versions of Chebotarev's theorem (e.g.~\cite[(3.2)]{Thorner}) guarantee that we will not be unlucky for too long. Note that the condition of $\gamma$ being an $e$-th power in $\F_v^\times$, with $q_v\equiv 1\bmod e$, can be easily checked using Euler's criterion: it holds if and only if $\gamma^{(q_v-1)/e}\equiv 1\bmod v$.

Once $v_i$ has been specified, we fix a residue class $b_i\in\F_{v_i}^\times$ that generates $\F_{v_i}^\times/\F_{v_i}^{\times e}$, 
for example a primitive root in $\F_{v_i}^\times$.

We also need a uniformiser $\pi_i\in\OO_{S\cup \{v_i\}}^\times$ for each of the places $v_i$, whose existence is guaranteed by Lemma \ref{lemma}. It can be found, for example, by computing generators $\{g_0,\ldots,g_{r+1}\}$ of $\OO_{S\cup \{v_i\}}^\times$ and solving the linear Diophantine equation $c_0v_i(g_0)+\cdots+c_{r+1}v_i(g_{r+1})=1$ for the exponents in $\pi_i=g_0^{c_0}\cdots g_{r+1}^{c_{r+1}}$, with $v_i(\cdot)$ the exponential valuation at $v_i$.

\subsection{Finding the places $w_i$} The place $w_i$ has to be in $T(S;e;\{\gamma_0,\ldots,\gamma_r\},u_i\pi_i)$. The element $u_i\in \OO_{S\cup\{v_i\}}$ is specified by $u_1=1$ and inductively for $i>1$, as explained in \S \ref{sec:choosing_places}. We can determine $u_i$ by linear algebra with the images of $\pi_1,\ldots,\pi_i\in \widehat{\OO}_S^\times$ in
\begin{equation*}
  (\F_{w_1}^\times\times\cdots\times\F_{w_{i-1}}^\times)\otimes\Z/e\Z
\simeq (\Z/e\Z)^{i-1}. 
\end{equation*}
To specify the above isomorphism, fix elements $b_j'\in\F_{w_j}^\times$,
each generating $\F_{w_j}^\times/\F_{w_j}^{\times e}$,
and identify
\begin{equation*}
  (b_1'^{a_1},\ldots,b_{i-1}'^{a_{i-1}})\otimes 1\quad\text{ with }\quad(a_1,\ldots,a_{i-1})\in(\Z/e\Z)^{i-1}.
\end{equation*}
Let $l_{m,j}'\in\Z/e\Z$ be such that $\pi_j^{-1}\bmod w_m$ and $(b_m')^{l_{m,j}'}$ differ only by an $e$-th power in $\F_{w_m}^\times$. For example, if $b_m'$ is a primitive root of $\F_{w_m}^\times$, then $l_{m,j}'$ can be taken as the discrete logarithm of $\pi_j^{-1}$  with respect to the generator $b_m'$ of $\F_{w_m}^\times$, modulo $e$. Recalling how the uniformiser $\pi_j\in K_{v_j}^\times$ is embedded into $\widehat{\OO}_S^\times$, we see that the image of $\pi_j$ in $(\Z/e\Z)^{i-1}$ is then $(l_{1,j}',\ldots,l_{i-1,j}')$. Hence, the exponents in $u_i=\pi_1^{-c_1}\cdots\pi_{i-1}^{-c_{i-1}}$ can be determined by solving the $\Z/e\Z$-linear system
\begin{equation}\label{eq:ui_equations}
  \begin{pmatrix}
    l_{1,1}' & \cdots & l_{1,i-1}'\\
    \vdots & & \vdots\\
    l_{i-1,1}' & \cdots & l_{i-1,i-1}'
  \end{pmatrix}\cdot
  \begin{pmatrix}
    c_1\\\vdots\\c_{i-1}
  \end{pmatrix}=
  \begin{pmatrix}
    l_{1,i}'\\\vdots\\l_{i-1,i}'
  \end{pmatrix}.
\end{equation}
Once $u_i$ has been computed, we can again find $w_i\in T(S;e;\{\gamma_0,\ldots,\gamma_r\},u_i\pi_i)$ by an exhaustive search. In addition to \eqref{eq:search_cond_v}, it has to satisfy that the reduction of  $X^e-u_i\pi_i$ is irreducible over $\F_{w_j}$. Actually, we only need the potentially weaker (if $4\mid e$) condition that $(u_i\pi_i\mod w_i)\otimes 1$ generates $\F_{w_i}^\times\otimes \Z/e\Z$. As explained in the proof of Lemma \ref{lem:irreducibility_enough}, it is enough to verify that
\begin{equation*}
u_i\pi_i\bmod w_i \notin \F_{w_i}^{\times p}\quad\text{ for all primes $p\mid e$},
\end{equation*}
which can again be checked by Euler's criterion. 

As the conditions for the $w_i$ are more restrictive than for the $v_i$, we can potentially get smaller places by searching for $w_j$ as soon as possible, i.e.~as soon as $v_{j}$ has been determined, instead of computing first all the $v_i$ and then the $w_i$.

\subsection{Computing $\Phiprimenew$} In addition to the $l_{m,j}'$  already computed, let $l_{m,j}\in\Z/e\Z$ be such that $\pi_j^{-1}\bmod v_m$ differs from $b_m^{l_{m,j}}$ only by an $e$-th power in $\F_{v_m}^\times$, and $l_{m,j}=0$ if $j=m$.

Let us identify our chosen roots $b_i\in \F_{v_i}^\times$ with the basis elements $(1,\ldots,1,b_i,1\ldots,1)\otimes 1 \in G'$, and similarly for the $b_i'$. We call the resulting basis $b_1,b_1',\ldots,b_{\kprimenew},b_{\kprimenew}'$ the \emph{standard basis} of $G'$. Let us compare this basis to the basis $e_1,e_1',\ldots,e_{\kprimenew},e_{\kprimenew}'$ of $G'$, satisfying \eqref{eq:basis_condition}, that is constructed in the proof of Theorem \ref{thm:choose_places}. The matrix describing the change of basis from the standard basis to the $e_i,e_i'$ is
\begin{equation}\label{eq:change_of_basis}
  A =
  \begin{pmatrix}
    1       & l_{1,1}  & 0     & l_{1,2}  & \cdots & 0       & l_{1,\kprimenew}\\
    0       & l_{1,1}' & 0     & l_{1,2}' & \cdots & 0       & l_{1,\kprimenew}' \\
    0       & l_{2,1}  & 1     & l_{2,2}  & \cdots & 0       & l_{2,\kprimenew} \\
    0       & l_{2,1}' & 0     & l_{2,2}' & \cdots & 0       & l_{2,\kprimenew}' \\
    \vdots  & \vdots  & \vdots & \vdots  & \ddots & \vdots  & \vdots  \\
    0       & l_{\kprimenew,1}  & 0      & l_{\kprimenew,2} & \cdots & 1       & l_{\kprimenew,\kprimenew}  \\
    0       & l_{\kprimenew,1}' & 0      & l_{\kprimenew,2}'& \cdots & 0       & l_{\kprimenew,\kprimenew}'
 \end{pmatrix}.
\end{equation}
Let $B$ be the matrix representing the epimorphism $\Psi:G'\to(\Z/e\Z)^{\knew}$ from \eqref{eq:def_Psi} with respect to the bases $e_i,e_i'$ and $f_1,\ldots,f_{\knew}$, and $C$ a matrix representing $\Phinew:(\Z/e\Z)^{\knew}\to G$ with respect to $f_1,\ldots,f_{\knew}$ and any system $g_1,\ldots,g_{\knew}$ of generators for $G$. Then the epimorphism $\Phiprimenew:G'\to G$ is represented with respect to the standard basis of $G'$ and the generators $g_1,\ldots,g_{\knew}$ of $G$ by the matrix $R=CBA^{-1}$.

With the places $v_1,w_1,\ldots,v_{\kprimenew},w_{\kprimenew}$ defining $G'$, the standard basis $b_1,b_1',\ldots,b_{\kprimenew},b_{\kprimenew}'$ of $G$ and the matrix $R$ describing $\Phiprimenew$, we have now computed all the data required to describe the characteristic morphism $\rho_S^T:\widehat{\OO}_S^\times \to G$ from Theorem \ref{thm:main_construction}, and thus the induced epimorphism $\rho:I_K/K^\times\to G$ in a convenient way.

\subsection{Conductor, Artin map and  equations}\label{sec:prime_data}
The epimorphism $\rho:I_K/K^\times\to G$ induced by $\rho_S^T$ defines an extension $L_\rho/K$ together with an isomorphism $\Gal(L_\rho/K)\to G$. Let us see how to easily infer some information about $L_\rho$ from our description of $\rho_S^T$.

First of all, the places $v$ of $K$ that ramify in $L_\rho$ are those, for which the inertia group $\rho(\OO_v^\times)$ does not vanish in $G$. By our construction of $\rho_S'$ in \eqref{eq:rho_prime}, only places in $\{v_1,w_1,\ldots,v_{\kprimenew},w_{\kprimenew}\}$ can ramify. More precisely, the ramified places are all those $v_i,w_i$ for which the corresponding standard basis element $b_i$ or $b_i'$ is not in the kernel of $\Phiprimenew$. If we write $(r_1,\ldots,r_{\knew})^T$ for the corresponding column of $R$ and write $G$ additively, this means that $r_1g_1+\cdots+r_{\knew}g_{\knew}\neq 0$ in $G$. All these places are tamely ramified, so the conductor ideal of $L_\rho$ is just the product of the prime ideals corresponding to these places.

Every place $v\in S$ is split completely, as then $\rho(K_v^\times)$ vanishes in $G$. For places $v\notin S\cup\{v_1,w_1,\ldots,v_{\kprimenew},w_{\kprimenew}\}$, we can compute the Artin symbol $(L_\rho/K, v)\in G$ as follows: compute a uniformiser $\pi_v\in\OO_{S\cup\{v\}}^\times$ at $v$ and $l_{v,i}, l'_{v,i}\in\Z/e\Z$ such that $\pi_v^{-1}=b_i^{l_{v,i}}$ in $\F_{v_i}\otimes\Z/e\Z$ and $\pi_v^{-1}=(b_i')^{l_{v,i}'}$ in $\F_{w_i}\otimes\Z/e\Z$. Then the coordinates of $\rho_S'(\pi_v)\in G'$ with respect to the standard basis are given by $l_v = (l_{v,1},l_{v,1}',\ldots,l_{v,\kprimenew},l_{v,\kprimenew}')^T$. Hence,  the Artin symbol is given by $(L_\rho/K, v) = \rho(\pi_v)= (g_1,\ldots,g_{\knew})\cdot R\cdot l_v$. If $v\in\{v_1,w_1,\ldots,v_{\kprimenew},w_{\kprimenew}\}$ is unramified, we can compute $(L_\rho/K, v)$ in the same way as above, except that we take $l_{v,i}=0$ if $v=v_i$ and $l_{v,i}'=0$ if $v=w_i$.

Given the conductor $\mathfrak f$ and a method to evaluate the Artin symbol $(L_\rho/K,v)$, we have enough information to specify the congruence subgroup of the ray class group modulo $\mathfrak f$ corresponding to $L_\rho$ through an exhaustive search for generating splitting primes. From this data, one can compute polynomials defining $L_\rho$ using algorithms of computational class field theory, see, e.g. \cite{Cohen1,CohenBook,Fieker}. Such algorithms are implemented, for example, in the computer algebra system Magma.

\subsection{Subextensions} If $\Pi:G\to H$ is an epimorphism, then $\Pi\circ\rho$ describes a subextension $L_\Pi$ of $L_\rho$ together with an isomorphism $\Gal(L_\Pi/K)\to H$. The epimorphism $\Pi\circ\rho$ is induced by the characteristic morphism $\Pi\circ\rho_S^T:\widehat{\OO}_S^\times\to H$, and, in a similar way as described in \S \ref{sec:prime_data}, we can compute the ramified primes, splitting primes and equations for $H$.

Given a presentation $G\simeq\Z/n_1\Z\times\cdots\times\Z/n_{r}\Z$, we can
apply this in particular to the projections $\Pi_i:G\to H_i=\Z/n_i\Z$. Then
$L_\rho$ is the compositum of the fields $L_{\Pi_i}$, which we can exploit, for
example, to get equations for $L_\rho$ with lower computational effort.

\subsection{Elements with prescribed norms}
Having determined concrete polynomials defining the number field $L=L_\rho$,
one may want to find elements $\beta_1,\ldots,\beta_t\in L$ such that
$\N_{L/K}(\beta_i)=\alpha_i$ for $1\leq i\leq t$. By our construction, such
$\beta_i$ will exist, and one could find them through an exhaustive search. It
is much faster to use more sophisticated algorithms, see,
e.g. \cite{Simon}, \cite[\S 7.5]{CohenBook}, which solve norm equations
via $S$-units and are implemented in Magma.

\section{Illustrations}\label{sec:illustrations}

\subsection{A biquadratic extension of $\Q$ with norm $37/16$} In this first example, we take $G=(\Z/2\Z)^2$ and find an epimorphism $\rho:I_\Q/\Q^\times\to G$ with $\alpha=37/16\in\norm_{L_\rho/\Q}(L_\rho^\times)$.  

Due to the explicit nature of biquadratic fields, arguments simpler than ours would suffice here. A construction yielding distinct primes $p,q$ such that a given rational number $c$ is the norm of an element of $\Q(\sqrt{p},\sqrt{q})$ was communicated to D.~Loughran and the first-named author by J.-L.~Colliot-Th\'el\`ene, in response to a question raised at the workshop \og Rational Points 2017\fg{} in Schloss Schney (Germany). Similarly to our approach here, the primes $p,q$ are chosen in such a way as to guarantee that the extension $\Q(\sqrt{p},\sqrt{q})/\Q$ satisfies the Hasse norm principle and $c$ is a local norm at every place.
\smallskip

Since this example should serve as a first illustration of our approach, we will nevertheless use our setup relying on characteristic morphisms. The set $S=\{\infty,2,37\}\subset\Omega_\Q$ satisfies \eqref{eq:S}, and $\OO_S^\times=\Z[1/2,1/37]^\times=\langle -1,2,37\rangle$.
We can take $e=\knew=2$ and $\Phinew:G\to G$ the identity. With $\kprimenew={2\choose 2}=1$, we are aiming at a set $T$ consisting of a pair $v,w$ of places. The places in $T(S;e;1;\{-1,2,37\})$ are those primes $p$ that satisfy
\begin{equation*}
  \left(\frac{-1}{p}\right)=\left(\frac{2}{p}\right)=\left(\frac{37}{p}\right)=1,
\end{equation*}
with $\left(\frac{\cdot}{\cdot}\right)$ the Legendre symbol. We take $v=41$, the smallest such prime. In addition to the above, the place $w$ then has to satisfy
\begin{equation*}
  \left(\frac{41}{p}\right)=-1,
\end{equation*}
and the smallest choice is $w=137$. Hence, we have found our auxiliary morphism
\begin{equation*}
  \rho_S':\widehat{\Z}_S^\times\to(\F_{41}^\times\times\F_{137}^\times)\otimes \Z/2\Z = G' \simeq(\Z/2\Z)^2=G.
\end{equation*}
By our construction, we get $\rho_S'(\Q_{41}^\times)=G'$, so we can take $\rho_S=\Phiprimenew\circ\rho_S'$ for, e.g., the isomorphism $\Phiprimenew:G'\to G$, that identifies $\F_{41}^\times\otimes\Z/2\Z$ with the first $\Z/2\Z$-factor of $G$ and $\F_{137}^\times\otimes\Z/2\Z$ with the second. We have thus achieved that ${\bigwedge}^2\rho(\Q_{41}^\times)={\bigwedge}^2G$, which confirms the validity of the Hasse norm principle for $\rho$. All elements of $\Z[1/2,1/37]^\times$, including in particular our $\alpha=37/16$, are local norms at all places, and therefore in $\norm_{L_\rho/\Q}(L_\rho^\times)$.
\smallskip

Which biquadratic extension does $\rho$ describe? Via the projections $\Pi_1, \Pi_2: G\to\Z/2\Z$ to the two $\Z/2\Z$-factors of $G$, we can describe $L_\rho$ as the compositum of two quadratic subfields with conductors  $41$ and $137$. Hence, $\rho$ defines the number field $L_\rho=\Q(\sqrt{41},\sqrt{137})$.

We know now that $37/16$ is a norm from this field. In other words, using the norm form coming from the basis $1,\sqrt{41},\sqrt{137},\sqrt{5617}$ of $L_\rho$, the equation
\begin{align*}
  x_1^4 - 82x_1^2x_2^2 + 1681x_2^4 - 274x_1^2x_3^2 - 11234x_2^2x_3^2 + 18769x_3^4 + 44936 x_1x_2x_3x_4&\\ - 11234x_1^2x_4^2 - 460594 x_2^2x_4^2 - 1539058 x_3^2x_4^2 + 31550689 x_4^4 &=37/16
\end{align*}
has rational solutions $(x_1,x_2,x_3,x_4)\in\Q^4$. Indeed, Magma's norm form equation functionality instantly finds the solution
\begin{equation*}
  \left(4449545, -\frac{1389743}{2}, \frac{760267}{2}, -\frac{118739}{2}\right).
\end{equation*}

\subsection{A $\Z/6\Z \times (\Z/3\Z)^3$-extension of $\Q(\sqrt{-47})$ from which $2+3\sqrt{-47}$ is a norm.}
Here we consider an example that is sufficiently generic to demonstrate all the features of our construction. Our computations are assisted by the computer algebra system Sage\footnote{\url{https://www.sagemath.org/}}. Towards the end, we will compute concrete polynomials giving our extension $L$. For this, we use the Magma Calculator\footnote{\url{http://magma.maths.usyd.edu.au/calc/}}. None of the computations took longer than a few seconds.

Take the number field $K=\Q(\sqrt{-47})$ with class number $5$. We construct an
extension $L/K$ with Galois group $G=\Z/6\Z\times(\Z/3\Z)^3$, such that
$2+3\sqrt{-47}\in\norm_{L/K}(L^\times)$.

The principal ideal $(2+3\sqrt{-47})$ factors as the product of two prime ideals of $\OO_K$, namely
\begin{equation*}
  (2+3\sqrt{-47})=\left(7, \frac{3+\sqrt{-47}}{2}\right)\left(61, \frac{21+\sqrt{-47}}{2}\right).
\end{equation*}
We take $S$ to consist of the archimedean place of $K$ and these two prime ideals. Then $\OO_S$ has class number one, so \eqref{eq:S} is satisfied. A quick computation reveals the following set of generators $\{\gamma_0,\gamma_1,\gamma_2\}$ for the $S$-units $\OO_S^\times$: $\gamma_0=-1$, $\gamma_1=2+3\sqrt{-47}$ and $\gamma_2=128+3\sqrt{-47}$.

Take $\Phinew:(\Z/6\Z)^4\to G$ to be the epimorphism mapping each factor $\Z/6\Z$ to the corresponding factor $\Z/6\Z$ or $\Z/3\Z$ of $G$ in the obvious way. With $\knew=4$, we are looking for $\kprimenew={4\choose 2}=6$ pairs of places $v_1,w_1,\ldots,v_6,w_6$ of $K$ to define our auxiliary morphism $\rho_S'$ from \eqref{eq:rho_prime}.

Our places need to be chosen from $T(S;6;1;\{\gamma_0,\gamma_1,\gamma_2\})$, so they must be in $\places\smallsetminus S$ and satisfy the conditions that
\begin{equation}\label{eq:example_cond_v}
  -1,\ 2+3\sqrt{-47}\ \text{ and }\ 128+2\sqrt{-47}\ \text{ are $6$-th power residues in $\F_v^\times$}.
\end{equation}
Here is a list of the (prime ideals corresponding to) the first few places satisfying these conditions, in order of the rational primes lying below:
\begin{align*}
 &(97, (27+\sqrt{-47})/2),\ 
 (569),\ 
 (809),\ 
 (1033),\ 
 (1381, (1445+\sqrt{-47})/2),\\
 &(1913),\ 
 (2281, (619+\sqrt{-47})/2),\ 
 (2377, (3677+\sqrt{-47})/2),\\
 &(2887),\ 
 (4621),\ 
 (4789, (2537+\sqrt{-47})/2),\ 
   (5227),\ 
   (6101).
\end{align*}
Hence, we take $v_1$ as the first of these places, choose $b_1$ to be a primitive root in $\F_{v_1}^\times$ and compute a uniformiser $\pi_1\in\OO_{S\cup \{v_1\}}^\times$ at $v_1$ as in \S \ref{sec:finding_vi}. The precise values are recorded in Table \ref{tab:v}.

\begin{table}[H]\centering
\begin{tabular}{|c|c|c|c|}
  \hline
  $i$ & $v_i$                         & $b_i$           & $\pi_i$               \\
  \hline
  1   & $(97, (27+\sqrt{-47})/2)$     & $5$             & $-353 + 48\sqrt{-47}$ \\
  2   & $(809)$                       & $1+\sqrt{-47}$  & $809$\\
  3   & $(1033)$                      & $20+\sqrt{-47}$ & $1033$\\
  4   & $(1913)$                      & $7+\sqrt{-47}$  & $1913$\\
  5   & $(2377, (3677+\sqrt{-47})/2)$ & $5$             & $712+81\sqrt{-47}$\\
  6   & $(2887)$                      & $1+\sqrt{-47}$  & $2887$\\
  \hline
\end{tabular}
\caption{The places $v_i$ with primitive roots $b_i$ and uniformisers $\pi_i$.}
\label{tab:v}
\end{table}

\begin{table}[H]\centering
\begin{tabular}{|c|c|c|c|}
  \hline
  $i$ &  $u_i$                 & $w_i$                           & $b_i'$          \\
  \hline
  1   &  $1$                  & $(569)$                          & $2+\sqrt{-47}$ \\
  2   &  $1$                  & $(1381, (1445 + \sqrt{-47})/2)$  & $2$ \\
  3   &  $\pi_2^{-3}$          & $(2281, (619 + \sqrt{-47})/2)$   & $7$ \\
  4   &  $\pi_2^{-5}\pi_3^{-4}$ & $(4789, (2537 + \sqrt{-47})/2)$ & $2$ \\ 
  5   &  $\pi_1^{-4}\pi_3^{-2}$ & $(4621)$                      & $5+\sqrt{-47}$\\
  6   & $\pi_1^{-4}\pi_2^{-5}\pi_3^{-5}\pi_4^{-2}\pi_5^{-2}$ & $(65 + 12\sqrt{-47})$ & $7$\\
  \hline
\end{tabular}
\caption{The places $w_i$ with factors $u_i$ and primitive roots $b_i$.}
\label{tab:w}
\end{table}

As $u_1=1$, the place $w_1$ has to satisfy in addition to \eqref{eq:example_cond_v} the condition that $X^6-\pi_1$ is irreducible over $\F_{w_1}^\times$, so $\pi_1$ may not be a quadratic or cubic residue in $\F_{w_1}^\times$. We choose $w_1$ as the first place satisfying these conditions and $b_1'$ as a primitive root for $\F_{w_1}^\times$. See Table \ref{tab:w} for the values.

Next, we pick $v_2$ as the next available place from our list and choose a corresponding primitive root $b_2$ and uniformiser $\pi_2\in\OO_{S\cup\{v_2\}}$. As $v_2$ is generated by an inert rational prime, the choice of uniformiser is simple.

Before choosing $w_2$, we need to find $u_2$. To this end, we compute the discrete logarithms $l_{1,1}'=5$ and $l_{1,2}'=0$ (modulo $6$) of $\pi_1^{-1}$ and $\pi_2^{-1}$ in $\F_{w_1}^\times$, with respect to the chosen primitive root $b_1'$. Hence, the images of $\pi_1$ and $\pi_2$ in $\F_{w_1}^\times\otimes\Z/6\Z$ have coordinates $5$ and $0$ with respect to the basis $b_1'\otimes 1$, and thus $\pi_2$ has the same image as $\pi_1^0$. We take $u_2=(\pi_1^0)^{-1}=1$. The condition that $w_2$ needs to satisfy, in addition to \eqref{eq:example_cond_v}, is that $X^6-\pi_2$ is irreducible over $\F_{w_2}$. We take $w_2$ as the first place that satisfies this condition and compute a primitive root $b_2'$ of $\F_{w_2}^\times$. 

We choose $v_3$ as the next available place in our list, together with a primitive root $b_3$ and uniformiser $\pi_3$.

For $u_3$, the $\Z/6\Z$-linear system \eqref{eq:ui_equations} takes the form
\begin{equation*}
  \begin{pmatrix}
   5 &0 \\
   4 &5 
  \end{pmatrix}\cdot
  \begin{pmatrix}
    c_1\\c_2
  \end{pmatrix}=
  \begin{pmatrix}
    0 \\ 3
  \end{pmatrix},
\end{equation*}
with solution $(c_1,c_2)=(0,3)$. Hence, we take $u_3=\pi_2^{-3}$, and choose $w_3$ as the smallest place that satisfies \eqref{eq:example_cond_v} and moreover that $X^6-u_3\pi_3$ is irreducible over $\F_{w_3}$.

We continue in the same fashion to compute the remaining places $v_4,w_4,v_5,w_5,v_6,w_6$ as well as the corresponding primitive roots and uniformisers (for the $v_i$). The data is collected in Tables \ref{tab:v} and \ref{tab:w}.

This gives us everything that we need to define the auxiliary morphism
\begin{equation*}
  \rho_S' : \widehat{\OO}_S^\times\to G'= (\F_{v_1}^\times\times\F_{w_1}^\times\times\cdots\times\F_{v_6}^\times\times\F_{w_6}^\times)\otimes\Z/6\Z
\end{equation*}
and the standard basis $b_1\otimes 1, b_1'\otimes 1,\ldots,b_6\otimes 1,b_6'\otimes 1$. The change of basis from the standard basis to the $e_i,e_i'$ satisfying \eqref{eq:basis_condition} is described by the matrix \eqref{eq:change_of_basis}, which takes in our case the form

\setcounter{MaxMatrixCols}{20}
\begin{equation*}
  A=\begin{pmatrix}
1 & 0 & 0 & 0 & 0 & 3 & 0 & 0 & 0 & 3 & 0 & 1 \\
0 & 5 & 0 & 0 & 0 & 0 & 0 & 0 & 0 & 2 & 0 & 0 \\
0 & 4 & 1 & 0 & 0 & 0 & 0 & 0 & 0 & 0 & 0 & 0 \\
0 & 4 & 0 & 5 & 0 & 3 & 0 & 1 & 0 & 4 & 0 & 0 \\
0 & 1 & 0 & 2 & 1 & 0 & 0 & 0 & 0 & 4 & 0 & 2 \\
0 & 0 & 0 & 2 & 0 & 1 & 0 & 2 & 0 & 2 & 0 & 5 \\
0 & 0 & 0 & 0 & 0 & 0 & 1 & 0 & 0 & 4 & 0 & 0 \\
0 & 1 & 0 & 0 & 0 & 1 & 0 & 5 & 0 & 0 & 0 & 1 \\
0 & 1 & 0 & 4 & 0 & 2 & 0 & 0 & 1 & 0 & 0 & 0 \\
0 & 2 & 0 & 2 & 0 & 2 & 0 & 4 & 0 & 5 & 0 & 4 \\
0 & 3 & 0 & 4 & 0 & 4 & 0 & 0 & 0 & 4 & 1 & 0 \\
0 & 4 & 0 & 1 & 0 & 5 & 0 & 5 & 0 & 0 & 0 & 3
  \end{pmatrix}.
\end{equation*}
An epimorphism $\Psi:G'\to(\Z/6\Z)^4$ as in \eqref{eq:def_Psi} is given, with respect to the basis $e_i,e_i'$ of $G'$ and the standard basis $f_1,f_2,f_3,f_3$ of $(\Z/6\Z)^4$ by the matrix
\begin{equation*}
B=\begin{pmatrix}
  1 & 0 & 1 & 0 & 1 & 0 & 0 & 0 & 0 & 0 & 0 & 0 \\
0 & 1 & 0 & 0 & 0 & 0 & 1 & 0 & 1 & 0 & 0 & 0 \\
0 & 0 & 0 & 1 & 0 & 0 & 0 & 1 & 0 & 0 & 1 & 0 \\
0 & 0 & 0 & 0 & 0 & 1 & 0 & 0 & 0 & 1 & 0 & 1
\end{pmatrix}.
\end{equation*}
The epimorphism $\Phinew:(\Z/6\Z)^4\to G$ just reduces the last three coordinates modulo $3$, so it is represented with respect to the $f_i$ and the standard  generators $(1,0,0,0)$, $(0,1,0,0)$, $(0,0,1,0)$, $(0,0,0,1)$ of $G$ by the identity matrix. In total, our map $\Phiprimenew=\Phinew\circ\Psi:G'\to G$ is given, with respect to the standard basis of $G'$ and the standard generators of $G$, by the matrix
\begin{equation*}
  R=BA^{-1}=\begin{pmatrix}
1 & 5 & 1 & 4 & 1 & 5 & 0 & 4 & 0 & 1 & 0 & 2 \\
0 & 2 & 0 & 2 & 0 & 2 & 1 & 0 & 1 & 2 & 0 & 2 \\
0 & 1 & 0 & 4 & 0 & 3 & 0 & 2 & 0 & 4 & 1 & 5 \\
0 & 4 & 0 & 4 & 0 & 1 & 0 & 4 & 0 & 1 & 0 & 0
            \end{pmatrix}.
          \end{equation*}
          We have now fully specified the characteristic morphism $\rho_S=\Phiprimenew\circ\rho_S':\widehat{\OO}_S^\times\to G$ that induces our epimorphism $\rho:I_K^\times/K^\times\to G$ giving the extension $L_\rho=L$.

          Let us find concrete polynomials defining $L$. To this end, we note that $L$ is the compositum of four subextensions $L_1,L_2,L_3,L_4$ corresponding to the projections $\Pi_j$ from $G$ to its cyclic factors $H_j=\Z/6\Z$ (for $j=1$) and $H_j=\Z/3\Z$ (for $j=2,3,4$).

          The extension $L_j$ is thus induced by the characteristic morphism $\rho_{S,j} = \Phiprimenew_j\circ\rho_S$, where the map $\Phiprimenew_j=\Pi_j\circ\Phiprimenew$ is represented by the $j$-th row $R_j$ of the matrix $R$. We can read off the places of $K$ ramified in $L_j$ as those $v_i,w_i$ whose entry in the $j$-th row of $R$ is not equal to $0$ (modulo $3$ for $j=2,3,4$). Hence, the conductor ideal of $L_j$ is $\mathfrak{f}_j$, with
          \begin{align*}
            \mathfrak{f}_1 &= v_1w_1v_2w_2v_3w_3w_4w_5w_6, \\
            \mathfrak{f}_2 &= w_1w_2w_3v_4v_5w_5w_6 ,      \\
            \mathfrak{f}_3 &= w_1w_2w_4w_5v_6w_6,   \\
            \mathfrak{f}_4 &= w_1w_2w_3w_4w_5.
          \end{align*}
          All places in $S$ split completely in all of the $L_j$. For a place $v\notin S$, we compute a uniformiser $\pi_v\in\OO_{S\cup\{v\}}^\times$ as in \S \ref{sec:finding_vi} and the vector $l_v=(l_{1,v},l_{1,v}',\ldots,l_{6,v},l_{6,v}')$, with $l_{i,v}$ (or $l_{i,v}'$) the discrete logarithm of $\pi_v^{-1}$ in $\F_{v_i}^\times$ (or $\F_{w_i}^\times$) with respect to $b_i$ (or $b_i'$), taken modulo $6$. Then $l_v$ consists of the coordinates of $\rho_S'(\pi_v)$ in the standard basis of $G'$, so $\rho_{S,j}(\pi_v)=0$ if and only if $R_jl_v=0$ (modulo $3$ for $j=2,3,4$). These $v$ are exactly the places not in $S$ that split completely in $L_j$.

In each of the extensions $L_j$, after an exhaustive search for places satisfying these conditions, we present the first few places of $K$ that split completely: for $L_1$, these are the prime ideals
\begin{align*}
  &(7, (3+\sqrt{-47})/2),
 (7,  (11+\sqrt{-47})/2),
 (53, (71+\sqrt{-47})/2),
 (59, (37+\sqrt{-47})/2),\\
 &(61, (21+\sqrt{-47})/2),
 (67),
 (97, (167+\sqrt{-47})/2),
 (103, (57+\sqrt{-47})/2),\\
 &(131, (79+\sqrt{-47})/2),
 (149, (129+\sqrt{-47})/2),
 (157, (253+\sqrt{-47})/2).            
\end{align*}
For $L_2$, we get the prime ideals
\begin{align*}
  &(2, (1+\sqrt{-47})/2),
 (5),
 (7, (3+\sqrt{-47})/2),
 (11),
 (41),
 (43),\\
 &(53, (71+\sqrt{-47})/2),
 (59, (37+\sqrt{-47})/2),
 (61, (21+\sqrt{-47})/2),\\
 &(61, (101+\sqrt{-47})/2),
 (67),
 (73),
 (79, (43+\sqrt{-47})/2).
\end{align*}
For $L_3$, the first few splitting places are
\begin{align*}
 &(2, (-1+\sqrt{-47})/2),
 (2, (1+\sqrt{-47})/2),
 (3, (1+\sqrt{-47})/2),\\
 &(7, (3+\sqrt{-47})/2),
 (17, (19+\sqrt{-47})/2),
 (29),
 (43),
 (53, (35+\sqrt{-47})/2),\\
 &(61, (21+\sqrt{-47})/2),
 (71, (109+\sqrt{-47})/2),
 (73),
 (79, (43+\sqrt{-47})/2).
\end{align*}
Finally, the first few splitting primes for $L_4$ are
\begin{align*}
 &(2, (-1+\sqrt{-47})/2),
 (2, (1+\sqrt{-47})/2),
 (5),
 (7, (3+\sqrt{-47})/2),\\
 &(17, (19+\sqrt{-47})/2),
 (19),
 (23),
 (29),
 (37, (45+\sqrt{-47})/2).
\end{align*}
In each case, the presented list of splitting primes is barely enough to generate the correct congruence subgroup of the ray class group modulo $\mathfrak{f}_j$. From this information, Magma obtains the following polynomials defining the $L_j$:

The field $L_1$ is generated by roots of the polynomials
\begin{align*}
  &X^2 + 1625554186831677234132\sqrt{-47} + 1811860086730297979035
\end{align*}
and
\begin{align*}
  X^3 - &3748167037906625162481707496423982188817923353276910868587 X\\
  + & 20932667006810986711572641003097069272613523\\
    &9089132839261455735058433435300040785407916\\
  - &4892889242489320858461993956906597337627420\\
    &962014982904896214316984337417374211250173\sqrt{-47}.
\end{align*}
The field $L_2$ is generated by a root of
\begin{align*}
  X^3 - &514252500054134587252291124378250069242533908387 X\\
     +\frac{1}{2} \Big(- &298646938917972463681140245645712648\\
        &301029855802019621285424376904094635\\
  + &1662861756262325040866728524283288871452\\
    &494124982673546015261665625225\sqrt{-47} \Big).
\end{align*}
For $L_3$, we get the generating polynomial
\begin{align*}
  X^3 - &913736091749824689643505040786706198418283 X\\
      - & 233073217500727285419128281222253866235433000542376262114304064\\
      + & 22955375656018471299613470304327792587793595929519410170152870\sqrt{-47},
\end{align*}
and for $L_4$ the polynomial
\begin{align*}
  X^3 - &4720036166349902544210196434745323 X\\
  - &142867906955019411285258390809644752015601994134080\\
  - &6327448763258896081623032545540797957509248167898\sqrt{-47}.
\end{align*}
Hence, the compositum $L$ is generated over $K$ by roots of all of these polynomials.

\end{document}